\documentclass[a4paper]{amsart}%
\usepackage{amsmath}%
\usepackage{amsfonts}%
\usepackage{amssymb}%
\usepackage{graphicx}
\usepackage{color}
\usepackage{amscd}
\usepackage{mathrsfs}
\usepackage{amsthm}
\usepackage{verbatim}
\usepackage{soul}
\usepackage{hyperref}

\newtheorem{theorem}{Theorem}

\newtheorem{definition}[theorem]{Definition}

\newtheorem{lemma}[theorem]{Lemma}

\newtheorem{remark}[theorem]{Remark}

\newcommand{\N}{\hbox{\ensuremath{\mathbb{N}}}}

\newcommand{\R}{\hbox{\ensuremath{\mathbb{R}}}}

\newcommand{\hau}{\mathcal H}
\newcommand{\pac}{\mathcal P}

\begin{document}

\title[Hausdorff and Packing measures of Cantor sets]{Exact Hausdorff and Packing measure of certain Cantor sets, not necessarily self-similar or homogeneous}
\author{Leandro Zuberman} 
\address{Departamento de Matem\'atica, FCEN-UNMdP, Dean Funes 3350\\ B7602AYL - Mar del Plata - BA - Argentina
}
\email{leandro.zuberman@gmail.com}

\begin{abstract}
We compute the exact Hausdorff and Packing measures of linear Cantor sets which might not be self-similar or homogeneous . The calculation is based on the local behavior of the natural probability measure supported on the sets.
\end{abstract}

\date{}
\maketitle


\section{Introduction}

Relevant tools in the study of sets of null Lebesgue measure are the Hausdorff and Packing dimensions and measures. During the last decades an enormous body of literature has been developed in order to estimate dimension of sets. However, much less is known about the exact value of the respective measures. 

In the last years, the exact measure has attracted the attention of the community. Very general results have been demonstrated by Olsen who has computed the exact Hausdorff and Packing measure of self-similar sets satisfying the open set condition \cite{Ols08}. Local densities are the main tool used by Olsen. 
Mor\'an and Llorente have established results in the same direction and made considerable progress developing  algorithms to approximate these measures \cite{Mor05, LM12, LM16}. 

For linear Cantor sets, the first result known, is due to Marion who established the exact Hausdorff measure for self-similar Cantor sets satisfying the open set condition \cite{Mar85}.
Ayer and Strichartz \cite{AS90} computed the exact Hausdorff measure of self-similar Cantor sets satisfying weaker conditions than the open set condition and gave very precise algorithms to obtain it. Qu et al \cite{QRS03}, have determined the Hausdorff measure of homogeneous (meaning that the gaps have all the same length) Cantor sets including the family considered by Marion and some non self-similar sets. Recently, Pedersen et al \cite{PP13} have partially extended the result by Qu et al, computing the exact Hausdorff measure of a family of Cantor sets not previously considered, including non-homogeneous ones. However, some sets considered by Qu et al do not satisfy the hypothesis required in \cite{PP13}. In the present article, we consider Cantor sets which might not be self-similar or homogeneous, including the considered in \cite{QRS03} and in \cite{PP13} as well as examples not considered before and establish a formula for the exact Hausdorff measure (see Theorem \ref{Hausdorff}). The hypothesis assumed are of separation type and decay on the lengths of gaps. 

The results for Hausdorff measure mentioned in the previous paragraph, have its counterpart for Packing measures. For the one third classical Cantor set, Feng et al computed the exact Hausdorff measure, based on the lower density \cite{FHW00}. For self-similar sets satisfying the open set condition, the exact Packing measure was established in \cite{Fen03, QZJ04}. Baek \cite{Bae06}, continuous their work giving some estimates on Packing measures for non self-similar sets . Garc\'\i a \cite{GZ13} et al gave the exact Packing measures for homogeneous Cantor sets that are not necessarily self-similar. In this paper we establish a formula generalizing the previous result but including non self-similar or homogeneous sets (see Theorem \ref{packing}). The assumed hypothesis are of separation type and we also impose a bound on the number of sub-intervals (or children) at each step. While the gaps do not play any role in the case of the Hausdorff measure -provided a separation condition is satisfied- they are crucial for computing the Packing measure. This made the calculationof Packing measure in this case -with gaps of variable length- much more technical.

The exact measure is closely related to the local behavior of the natural measure supported on the set. All results mentioned in the previous paragraphs use this fact. In the case of self-similar sets with a separation condition, the local behavior can be expressed as an absolute minimum or maximum. However, this form is not available for non self-similar sets and the limiting process can not be avoided. This is probably the reason why the algorithms are only proposed for self-similar sets. 

It is worth mentioning, that in a different direction, Qiu has computed the exact Hausdorff and Packing measures of some self-similar sets with overlaps \cite{Qiu15} of finite type. 

\section{Notation and statement of results}
We start by defining the family of Cantor sets we are going to work with. We follow a product structure construction. As usual, we take collections of finite and nested closed intervals,step by step. The Cantor set will be the (limit) intersection of all these sets. In this case, the number of intervals at each step, the contraction ratio and the location of each interval are variable at each step.  

\subsection{Construction of Cantor sets}\label{setting}
Let us make precise what we explained in the previous paragraph. 
Fix a real sequence $(r_n)_{n\in\N}\subset (0,1/2)$, a sequence of integers $(m_n)_{n\in\N}\subset\N_{\geq 2}$ and a sequence of finite sets $D_n\subset \left[0,\frac{1-r_n}{r_n}\right]$, each with $m_n$ elements. These will yield the contraction radii, the number of children and the location of those children at each step. 

The elements of $D_n$ will be called digits and some assumptions will be assumed. The maximum and minimum in each $D_n$ will be prescribed and we introduce the following notation: \[D_n:=\{d^n_1:=0,d^n_2,\dots,d^n_{m_n-1},d^n_{m_n}:=\frac{1-r_n}{r_n}\}\]
We assume that $d^n_j$ increase with $j$ and $d^n_{j+1}-d^n_j>1$. Note that this condition together with the election of the first and last digits implies that $r_n m_n<1$. It will be useful to introduce the following sequences: $\mu(n)=m_1\dots m_n$ and $s_n=r_1\dots r_n$.

With this setting, we introduce the sets $P_n:=D_1s_1+\dots+D_ns_n=\{\sum_{j=1}^n d_js_j:d_j\in D_j\}$ and
\[C_n=P_n+[0,s_n]=
\bigcup_{x\in P_n}[x,x+s_n].\]

Finally, 
we define the Cantor set associated to the sequences $(r_n)$, $(m_n)$ and digits $D_n$ as the intersection of all $C_n$: 
\[C:=C((r_n),(m_n),(D_n))=
\bigcap_{n\in\N} C_n.\] 

Note that $C_n$ is the union of $\mu(n)$ intervals of lengths $s_n$. 
A \emph{basic interval} of order $n$ will be one of the intervals of $C_n$ and a \emph{simple interval} of order $n$ will be a union of consecutive basic intervals of order $n$.

Inside each basic interval of order $n-1$ there are $m_n$ basic intervals of order $n$. So, the difference between them are $m_{n}-1$ open intervals; we will call them gaps and note with $g^{n}_1,\dots,g^{n}_{m_{{n}-1}}$. We abuse the notation and use the same letter for the gap and its length: 
\begin{equation}
g^n_j=s_n(d^n_{j+1}-d^n_{j}-1)\text{ for }j=1,\dots,m_n-1. 
\label{gaps}
\end{equation}

\subsection{Finite and infinite words}

As usual, words are useful to designate element and intervals in Cantor sets. For $n\geq 1$ define
\[\Sigma_n=\{i_1\dots i_n:1\leq i_j\leq m_j , \ 1\leq j\leq n \},\]
 $\Sigma=\bigcup_{n\geq 1}\Sigma_n$ and $\overline{\Sigma}=\{i_1\dots i_n\dots :1\leq i_j\leq m_j , \ 1\leq j \}.$ 
 If $\sigma\in\Sigma_n$ the length of $\sigma$ will be $|\sigma|=n$. If $\sigma$ belongs to $\Sigma_n$ or $\overline{\Sigma}$ and $j<n$ we denote the truncation of $\sigma$ as $\sigma|j=\sigma_1\dots \sigma_j$. The concatenation of two words $\sigma\in\Sigma_n$ and $\tau\in\Sigma_k$ will be $\sigma\tau\in\Sigma_{n+k}$ and analogously if $\tau\in\overline\Sigma$. 

 Elements of $P_n$ and basic intervals of order $n$ are associated to elements of $\Sigma_n$ in the following way. Given $\sigma=i_1\dots i_n\in\Sigma_n$, we associate the element $x=\sum _{j=1}^n d^n_{i_j}s_j\in P_n$ and the interval $[x,x+s_n]$. This interval will be denoted $I_\sigma$. Note that the separation condition $d^n_{j+1}-d^n_j>1$ guarantees that the correspondence is biunivocal. We will also need to introduce a notation for the borders of the interval: $I_\sigma=[a(\sigma),b(\sigma)]$. Finally, if $x\in C$ there is a unique word  $\sigma\in\overline\Sigma$ such that $x\in I_{\sigma|n}$ for all $n\in\N$. We will denote $\sigma=\sigma(x)$. 
 
 \subsection{Natural measure on Cantor set} Following Kolmogorov's extension Theorem, there is a unique measure supported on $C$ and satisfying $\mu(I_\sigma)=\mu(n)^{-1}$ for all $\sigma\in\Sigma_n$. Whenever the Cantor set is fixed, the measure $\mu$ will denote this measure.

 \subsection{Statement of Results}
 With these notations, the Hausdorff and Packing dimension of the Cantor set are well known. Actually, (see for example \cite{Fal90}):
 \[\dim_H C= \liminf_{n\to\infty}\frac{\log\mu(n)}{\log s_n}, \qquad \dim_P C= \limsup_{n\to\infty}\frac{\log\mu(n)}{\log s_n}.\]
 The objective of this paper is to determine the exact measures. 
 We state here the Theorems which will be proved in the upcoming sections. We will assume the following separation condition:
 \begin{equation}\label{separation}  d^n_{j+1}-d^n_j\geq c>1\qquad \forall n\geq1\quad \forall j=1,2,\dots,m_n.\end{equation}

  \begin{theorem}\label{Hausdorff}
With the above notations, suppose \eqref{separation} is satisfied and 
for any $J_1,J_2\subset \{1,\dots,m_n-1\}$, with $\#J_1\geq \#J_2$ and any $j=1,\dots,m_{n-1}-1$, suppose that:\begin{equation}\label{cond2}
 \frac{\sum_{t\in J_2} g_t^n}{g^{n-1}_j+\sum_{t\in J_1}g^n_t}\leq 1-\left(\frac{\#J_1-\#J_2+1}{\#J_1+1}\right).
 \end{equation}
If $s=\liminf_{n\to\infty}\frac{\log\mu(n)}{\log s_n}$
then the $s$-Hausdorff measure of $C$ is: \newline $\hau^s(C)=\liminf_{n\to\infty}\mu(n) s_n^s$. 
   
  \end{theorem}
  
In the preceding formula, we see that the length of gaps is not relevant for the Hausdorff measure. In contrast, it will be very important in the case of the Packing measure. The notation can become very technical, so we introduce it before stating the Theorem. Each constant corresponds to left, right or center approximation to the ratio between the measure of basic interval and its length.  

Fix $t:=\limsup_{n\to\infty}\frac{\log\mu(n)}{\log s_n}$. For each $n\geq 1$ there is a $k_n\in\{1,2,\dots,m_n-1\}$ where the following maximum is reached:
\[\alpha_n:=\max_{1\leq k\leq m_n-1}\left\{\frac{\mu(n)}{k}\left( ks_n+\sum_{i=1}^k g^n_i\right)^t\right\}=\frac{\mu(n)}{k_n}\left( k_ns_n+\sum_{i=1}^{k_n} g^n_i\right)^t\]
Also define $\alpha=\limsup_{n\to\infty}\alpha_n$. 

On the other side:
\begin{eqnarray*}
 \beta_n:&=&\max_{2\leq k\leq m_n}\left\{\frac{\mu(n)}{m_n-k+1}\left( (m_n-k+1)s_n+\sum_{i=k}^{m_n-1} g^n_i\right)^t\right\}\\
 &&=\frac{\mu(n)}{m_n-k_n+1}\left( (m_n-k_n+1)s_n+\sum_{i=k_n}^{m_n-1} g^n_i\right)^t
 \end{eqnarray*}

and $\beta=\limsup_{n\to\infty}\beta_n$.

Finally, for each $n\geq 1$ such that $m_n\geq3$ take $k_n^1,k_n^2\in\{2,3,\dots,m_n-1\}$ such that:
\begin{eqnarray*}
 \max_{2\leq k^1<k^2<m_n}\left\{\frac{\mu(n)}{k^2-k^1+1}\left((k^2-k^1+1)s_n+\sum_{i=k^1-1}^{k^2}g_i^n\right)^t\right\}\\=
\frac{\mu(n)}{k_n^2-k_n^1+1}\left((k_n^2-k_n^1+1)s_n+\sum_{i=k_n^1-1}^{k_n^2}g_i^n\right)^t.
\end{eqnarray*}

In the subsequence of integers $n$ such that $m_n\geq 3$ define:
 \begin{eqnarray*}
a_n: &=& b(k^1_1k^1_2\dots k^1_{n-1}(k^1_n-1)),\\ 
b_n & =  & a(k^2_1k^2_k\dots(k^2_n+1)),\\ 
d_n&=& d(C,(a_n+b_n)/2) \text{  and  }\\
\gamma  &=& \limsup_{n\to\infty} \frac{\mu(n)}{k^2_n-k^1_n+1}\left(b_n-a_n-2 d_n\right)^t.
 \end{eqnarray*}
 \begin{theorem}\label{packing}
 With the preceding notation, if $m_n\leq M$ for all $n\geq1$ and \eqref{separation} is satisfied, the $\pac^t(C)=\max\{2^{t}\alpha,2^{t}\beta,\gamma\}$. 
\end{theorem}

 When $m_n\geq 3$ for only finite many $n$, then $\gamma$ is not defined and is omitted in the Theorem.

\section{Hausdorff Measure}

In this section, we will compute the Hausdorff measure of the Cantor set. This estimate is  based on the local behaviour of the measure $\mu$ expressed by a generalization of the mass distribution principle instead of the upper density (see Lemma \ref{Frostman} below). We follow the approach of \cite{QRS03,PP13}.

Recall that  $s:=\liminf_{n\to\infty}\frac{\log \mu(n)}{|\log s_n|}$ and define \(B_s:=\liminf_{n\to\infty}\mu(n) s_n^s\). 

\begin{lemma}\label{Frostman}
Under the hypothesis of Theorem \ref{Hausdorff}.

For any $\varepsilon>0$ there is an $n_0$ such that $$\mu(P)\leq (B_s-\varepsilon)^{-1}|P|^s$$ for any simple interval $P$ of order $n> n_0$ contained in a basic interval of order $n_0$, where $|P|$ indicates the diameter of $P$.
\end{lemma}

\begin{proof}
Given $\varepsilon>0$ pick $n_0$ such that $\mu(n)s_n^s\geq (B_s-\varepsilon)$ for all $n\geq n_0$. 

If $P$ is basic of order $n\geq n_0$, say $P=I_\sigma$ for some $\sigma\in \Sigma_n$ then 
\[\mu(I_\sigma)=\mu(n)^{-1}\leq (B_s-\varepsilon)^{-1}s_n^s=(B_s-\varepsilon)^{-1}|I_\sigma|^s.\]
the thesis is valid. 

If $P$ is simple but not basic, the proof will be by induction on the order of $P$, say $n$.

\textit{Case 1: $n=n_0+1$.} Suppose $P$ is simple of order $n=n_0+1$ is not basic but is contained in a basic interval of order $n_0$. Then $P=[x,y]$ with $x=a(\sigma)$ and $y=b(\tau)$ with $\sigma,\tau\in \Sigma_n$ and $\sigma|n_0=\tau|n_0$.  Put $i:=\tau_n-\sigma_n$. Since $P$ is not basic, $i\geq 2$. We have:
\begin{eqnarray*}
 |P| &=& s_n\left(1+d_{\tau_n}-d_{\sigma_n}\right)=\\
 &=& s_{n_0} \left(1+r_n\left(d_{\tau_n}-d_{\sigma_n} \right)-(1-r_n)-r_n\right)\\
 &=& s_{n_0}\left(\frac{i}{m_n-1}+r_n\frac{m_n-1-i}{m_n-1}+A\right),
\end{eqnarray*}
where \begin{eqnarray*}
A  &=& \frac{m_n-1-i}{m_n-1}-r_n\frac{m_n-1-i}{m_n-1}+r_n\left(d_{\tau_n}-d_{\sigma_n} \right)-(1-r_n)-r_n\\
&=& \frac{m_n-1-i}{m_n-1} (1-r_n)+r_n(ci-1+r_n)\geq 0.
\end{eqnarray*}
Then, using concavity and the definition of $n_0$,
\begin{eqnarray*}
 |P|^s &\geq &s_n^s\left(\frac{m_n-1-i}{m_n-1}\right)+s_{n_0}^s\frac{i}{m_n-1}\\
 &\geq & (B_s-\varepsilon)\mu(n)^{-1}\left(\frac{m_n-1-i}{m_n-1}+\frac{i}{m_n-1}m_n
 \right)\\
 & = & (B_s-\varepsilon)\mu(n)^{-1}(i+1).
\end{eqnarray*}

\textit{Case 2: Inductive hypothesis} Assume that the thesis is valid for $n-1$ and we'll try to prove for $n$. 

If $P$ is a simple interval of order $n$ contained in a basic interval of order $n-1$, the proof is similar to case 1. So, we can assume $P$ is a simple interval of order $n$ not contained in a basic interval of order $n-1$. Say $P=[x,y]$ with $x=a(\sigma)$, $y=b(\tau)$ for some $\sigma,\tau\in \Sigma_n$ and $\sigma|n-1\neq \tau|n-1$. 

\textit{Case 2a: $\sigma_n=1$} In this case, $a(\sigma)=a(\sigma|n-1)$. Observe that $\tau_n\neq m_n$, otherwise $P$ would be of order $n-1$ or lower. Define $\hat\tau\in D_{n-1}$ such that $\hat\tau|n-2=\tau|n-2$ and $\hat\tau_{n-1}=\tau_{n-1}$, what means that $I_{\hat\tau}$ is the basic interval of order $n-1$ to the left of $I_{\tau|n-1}$. 

Now define $\lambda:=\displaystyle\frac{b(\tau|n-1)-b(\tau)}{b(\tau|n-1)-b(\hat\tau)}$.We have $0<\lambda<1$. 

Note that, for some $j=1,\dots,m_{n-1}-1$, we have:
\begin{eqnarray*} 
 \lambda &=& \frac{(m_n-\tau_n)s_n+\sum_{t=\tau_n}^{m_n-1}g^n_t}{m_ns_n+g^{n-1}_{j}+\sum_{t=1}^{m_n-1}g^n_t}\\
 &\leq & \max\{\frac{m_n-\tau_n}{m_n},\frac{\sum_{t=\tau_n}^{m_n-1}g^n_t}{g^{n-1}_{j}+\sum_{t=1}^{m_n-1}g^n_t},\}
 \end{eqnarray*}
so by \eqref{cond2} we get $\lambda\leq 1-\frac{\tau_n}{m_n}$.

Now, using the definition of $\lambda$ and the concavity of the exponential:
\begin{eqnarray*}
|P|^s&=& (y-x)^s=\left(\lambda(b(\hat\tau)  -x)+(1-\lambda)(b(\tau|n-1)-x)\right)^s\\
   &\geq& \lambda( b(\hat\tau) -x)^s+(1-\lambda)(b(\tau|n-1)-x)^s.
\end{eqnarray*}
Note that the intervals $[x,b(\tau|n-1)]$ and $[x,b(\hat\tau) ]$ are simple of order $n-1$ and we can use the inductive hypothesis. So:
\begin{eqnarray*}
|P|^s &\geq &    (B_s-\varepsilon)\left(\lambda\mu([x,b(\hat\tau) ])+(1-\lambda)\mu([x,b(\tau|n-1])\right)\\
& = & (B_s-\varepsilon)\left( \mu([x,b(\hat\tau) ])+(1-\lambda)\mu([b(\hat\tau) ,b(\tau|n-1])\right).
\end{eqnarray*}
Since $\mu([b(\hat\tau),b(\tau|n-1])=\mu(n-1)^{-1}$ and $1-\lambda \geq \frac{\tau_n}{m_n}$ we conclude that
\[|P|^s\geq (B_s-\varepsilon) \left( \mu([x,b(\hat\tau)])+\frac{\tau_n}{\mu(n-1) m_n}\right)=(B_s-\varepsilon)\mu([x,y]).\] 

\textit{case 2b:$\sigma_n\neq 1$} It is possible to construct a simple interval $Q$ of the same order as $P$, with the same number of basic intervals of order $n$ inside, with $|Q|\leq |P|$ and such that the left end point of $Q$ coincides with the left end point of an interval of order $n-1$. 

Suppose $\sigma_n>\tau_n$. The left point of $Q$ will be $a(\sigma')$ where $\sigma'\in\Sigma_{n-1}$ is such that $I_{\sigma'}$ is to the right to $I_{\sigma|n-1}$. The right point of $Q$ will be $b(\tau')$ where $\tau'\in \sigma_n$ is such that $\tau'|n-1=\tau|n-1$ and $\tau_{n}'=\tau_{n}+m_n-\sigma_n$. 

If $\sigma_n\leq \tau_n$, take $Q=[a(\sigma),b(\tau')]$ where $\tau'\in \Sigma_n$ is such that $\tau'|n-1=\tau|n-1$ and $\tau_{n}'=\tau_{n}-\sigma_n+1$.

See Lemma 5.2 in \cite{PP13} for details. 

That is, $Q$ is in case 2a and $\mu(Q)=\mu(P)$. Then, 
\[|P|^s\geq |Q|^s\geq (B_s-\varepsilon)\mu(Q)=(B_s-\varepsilon)\mu(P).\] 
\end{proof}

Now, Theorem \ref{Hausdorff} is a corollary of the previous Lemma \ref{Frostman}. 

\begin{proof}[Proof of Theorem \ref{Hausdorff}]
Since at each step, basic intervals are a cover of $C$, it is immediate that $\hau^s_\delta(C)\leq \mu(n)s_n^s$, if $s_n\leq \delta$. In consequence, $\hau^s(C)\leq B_s$. 

To prove the opposite inequality, pick $(U_i)_{i\geq 1}$ a $\delta$-cover of $C$. Without loss of generality, we can assume that each $U_i$ is contained in a simple interval $P_i$ with $\inf P_i=\inf U_i$ and $\sup P_i=\sup U_i$. Fix $\varepsilon>0$ and let $n_0$ as in Lemma \ref{Frostman}. If $\delta$ is small enough, $P_i$ can be assumed to be contained in a basic interval of order $n_0$. So, 
\[1=\mu(C)\leq \sum_{i}\mu(U_i)\leq \sum_i \mu(P_i) \leq (B_s-\varepsilon)^{-1}\sum |P_i|^s=(B_s-\varepsilon)^{-1}\sum |U_i|^s.\]
Since $(U_i)$ was arbitrary, we can take infimum over all possible covering and then let $\delta$ tends to zero, which gives the needed inequality $\hau^s(C)\geq B_s-\varepsilon$ for any $\varepsilon>0$.

\end{proof}

\begin{remark}
\begin{enumerate}
\item If $D_n=\{j\cdot d_n:j=0,1,\dots,m_n-1 \}$ for $d_n=\frac{1-r_n}{r_n(m_n-1)}$, then condition \ref{cond2} is equivalent to decreasing gaps. If we assume this condition of decreasing gaps, then we can not have any gap with length zero, and then condition \ref{separation} is also satisfied. Therefore, the main result in \cite{QRS03} is a corollary of Theorem \ref{Hausdorff}. 
\item In \cite{PP13}, Hausdorff measure is computed under the hypothesis:
\[m_nr_n^s<1\ \ \text{and} \ \ \  d_i^n-d_j^n\geq \max\{2,(1+i-j)^{1/s}-1\}.\]
These hypothesis are satisfied under the assumptions of Theorem \ref{Hausdorff}. So, Theorem 2.3 in \cite{PP13} can be obtained as a corollary of Theorem \ref{Hausdorff}
\end{enumerate}
\end{remark}

\section{Packing Measure}

The computation of the Packing measure is also based on the local behaviour of the measure $\mu$, but instead of proving a mass distribution principle we will compute the lower local density (for definition, see \ref{density} below)
of the measure $\mu$, the natural measure supported on the Cantor set. 

First, note that $\mu$ and $\pac^t\llcorner_C$, the restriction of the Packing measure to the Cantor set $C$, are closely related. In fact, for $\sigma\in\Sigma_k$, $\mu(I_\sigma)=\mu(k)^{-1}$ and  if $\pac^t(C)<\infty$, $\pac\llcorner_C(I_\sigma)$ only depends on $k$, and in consequence, there is a constant $\kappa$ such that $\mu=\kappa \pac^t\llcorner_C$. 

 Recall the definition of lower density of a measure. 
 \begin{definition}\label{density}
  If $\nu$ is a measure on $\R^d$ and $\alpha\geq 0$, the $\alpha$-lower density of $\nu$ at $x\in\R^d$ is:
  \(\Theta^\alpha(\nu,x):=\liminf_{r\to 0^+}(2r)^{-\alpha}\nu(B(x,r))\).
 \end{definition}
If $\pac^\alpha(A)<\infty$, then we have $\Theta^\alpha(\pac^\alpha\llcorner_A,x)=1$ for $\pac^\alpha$-almost every $x\in A$ (see Theorem 6.10 in \cite{Mat95}). 
In consequence, $\Theta^t(\mu,x)=\kappa \Theta^t(\pac\llcorner_C,x)=1$, and $\pac^t(C)=\Theta^t(\mu,x)^{-1}$ ( if $\pac^s(C)<\infty$).

 The proof of Theorem \ref{packing} is then reduced to the computation of the lower density of $\mu$. We accomplish this separately in two parts. First, we prove the lower bound, which is valid under more general conditions. Next, the upper bound is established under the assumptions of Theorem \ref{packing}.
 
 First, we will prove a technical Lemma that will be used in the proofs.
 
 \begin{lemma}\label{salto}
  Under the hypothesis of Theorem \ref{packing}, there is an integer $L$ such that:
  \[k_ns_n+\sum_{i=1}^{k_n}g^n_i\leq g^{n-\ell}_t,\]
for all $n\geq 1$, $k_n\leq m_n-1$, $1\leq t \leq m_{n-\ell}$ and $1\leq \ell\leq L$.  \end{lemma}

\begin{proof}
 For the left hand side, we have:
 \begin{eqnarray*}
  k_ns_n+\sum_{i=1}^{k_n}g^n_i &= &
  s_{n-1}-s_n\sum_{i=k+1}^{m_n-1}(d^n_i-d^n_{i-1}-s_n\\
  &\leq& s_{n-1}\left(1-r_nc(m_n-1-k)-r_n\right)\\
  &\leq& s_{n-1}(1-r_n).
 \end{eqnarray*}
On the other side, 
\begin{equation*} 
g^{n-\ell}_ts_{n-\ell}(c-1)\geq s_{n-L}(c-1).\end{equation*} So, 
 it is enough to see that $s_{n-1}(1-r_n)\leq s_{n-L}(c-1)$. Equivalently, $(1-r_n)r_{n-1}r_{n-2}\cdots r_{n-L+1}\leq c-1.$ Since $r_n\leq 1/2$, $L=[|\log (c-1)|/\log 2]$ works. 
\end{proof}

Now, we prove the lower bound.
\begin{lemma}\label{lowbound}
 For every $x\in C$, $\Theta^t(\mu,x)\geq \min\{2^{-t}\alpha^{-1},2^{-t}\beta^{-1},\gamma^{-1}\}$. 
\end{lemma}

\begin{proof}
Given $\varepsilon>0$ there is $N$ such that if $n\geq N$ then:
 \begin{align*}
 \alpha^{-1} -\varepsilon &\leq \frac{k}{\mu(n)\left(ks_n+\sum_{i+1}^kg^n_i\right)^t}\qquad &&\forall 1\leq k\leq m_n-1\\
 \beta^{-1}-\varepsilon &\leq  \frac{m_n-k+1}{\mu(n)\left( (m_n-k+1)s_n+\sum_{i=k-1}^{m_n-1} g^n_i\right)^t} \qquad&&\forall 2\leq k\leq m_n\\
 \gamma^{-1}-\varepsilon &\leq  
 \frac{k^2-k^1+1}{\mu(n)\left((k^2-k^1+1)s_n+\sum_{i=k^1-1}^{k^2}g_i^n\right)^t}&&\forall 1<k^1<k^2<m_n.
\end{align*}
 Fix $x\in C$ and $r>0$. There is an $n_0$ such that
 \begin{equation}\label{fixword}
\exists \sigma\in\Sigma_{n_0}: I_\sigma\subseteq B(x,r)\qquad \text{ and } I_\tau \nsubseteq B(x,r)\quad \forall \tau\in\Sigma_n\text{ with },n<n_0. 
\end{equation}
 Ir $r$ is taken small enough, then we can have $n_0\geq N$. Note that there are at most $2m_{n_0}-2$ words satisfying \eqref{fixword} since there are at most two words $\tau\in \Sigma_{n_0-1}$ such that $I_\tau\cap B(x,r)\neq \emptyset. $
 
 We will distinguish between three cases:\newline
 case 1: There is $\sigma\in\Sigma_{n_0}$ satisfying \eqref{fixword} such that $\sigma_{n_0}=1$,\newline
 case 2: there is $\sigma\in\Sigma_{n_0}$ satisfying \eqref{fixword} such that $\sigma_{n_0}=m_{n_0}$ and\newline
 case 3: none of the above conditions are satisfied.

 {\bf Case 1}. We will call $\sigma$ the word satisfying \eqref{fixword} such that $\sigma_{n_0}=1$ and $\tau=\sigma|n_0-1$. Also define:
 \[J=\max\{1\leq j<m_{n_0}:I_{\tau j}\subset B(x,r)\}.\]
 If $r\leq Js_{n_0}+\sum_{i=1}^Jg^{n_0}_i$, then
 \[
 \frac{\mu(B(x,r))}{(2r)^t}
 \geq
 \frac{J} {\mu(n) 2^t\left( Js_{n_0}+\sum_{i=1}^Jg^{n_0}_i\right)^t}\geq 2^{-t}(\alpha^{-1}-\varepsilon).
 \]
 
 Suppose now, $r>Js_{n_0}+\sum_{i=1}^Jg^{n_0}_i$. In this case, $B(x,r)$ contains a portion of $I_{\tau (J+1)}$. To estimate the measure of this portion, define a sequence $(j_k)_{k\geq n_0}$ inductively, starting by $j_{n_0}=J$ and then:
 \begin{equation}\label{defseq}
 j_k=\max\{1\leq j\leq m_k-1:I_{\tau J j_{n_0+1}\dots j_k}\subset B(x,r)
 \}.
 \end{equation}
 If $x+r\notin C$, then there is $n_1=\min\{k:x+r\notin I_{\tilde\sigma}\forall\tilde\sigma\in\Sigma_k\}$. 
In this case, $x+r\leq a(\tau)+\sum_{i=n_0}^{n_1}(j_is_i+\sum_{\xi=1}^{j_i}g_\xi^i)$ and (since $a(\tau)\leq x$) we have
\begin{eqnarray*}
\frac{\mu(B(x,r))}{(2r)^t}
 &\geq&
\frac{ \sum_{i=n_0}^{n_1}(j_i\mu(i)^{-1})}  {2^t\left( \sum_{i=n_0}^{n_1}(j_is_i+\sum_{\xi=1}^{j_i}g_\xi^i  \right)^t }\\
& \geq& 
2^{-t}\min_{n_0\leq i \leq n_1} \frac{j_i\mu(i)^{-1}}{\left( j_is_i+\sum_{\xi=1}^{j_i}g^i_\xi  \right)^t}\geq 2^{-t}(\alpha^{-1}-\varepsilon). 
 \end{eqnarray*}
If $x+r\in C$, then the sequence in \eqref{defseq} is infinite. But we still have for any $n_1>n_0$, 
\[
x+r\leq b(\tau j_{n_0}j_{n_0+1}\dots j_{n_1})= a(\tau)+\sum_{i=n_0}^{n_1}(j_is_i+\sum_{\xi=1}^{j_i}g_\xi^i)+s_{n_1}. 
\]
Since we also have $a(\tau)\leq x$ and $s_{n_1}\to 0$ when $n_1\to\infty$, in a similar way, for $n_1$ large enough, we obtain:
\[
\frac{\mu(B(x,r))}{(2r)^t}
 \geq
 2^{-t}(\alpha^{-1}-2\varepsilon). 
 \]

 {\bf Case 2} In a similar way, we obtain $
\displaystyle \frac{\mu(B(x,r))}{(2r)^t}
 \geq
 2^{-t}(\beta^{-1}-2\varepsilon).$
 
 {\bf Case 3} Suppose that $I_\sigma\subset B(x,r)$ with $\sigma\in\Sigma_{n_0}$ implies that $\sigma_{n_0}$ is not either 1 or $m_{n_0}$. If 
 $\liminf_{n\to\infty}
\frac{\mu(B(x,r))}{(2r)^t}
 \geq \min\{2^{-t}\alpha^{-1},2^{-t}\beta^{-1}\}$, the proof is concluded. Then assume $\liminf_{n\to\infty}\frac{\mu(B(x,r))}{(2r)^t}
<\min\{2^{-t}\alpha^{-1},2^{-t}\beta^{-1}\}$. 
Define:
\[j=\min\{1\leq i\leq m_n:I_{\tau i}\subset B(x,r)\}\text{ and }  J=\max\{1\leq i\leq m_n:I_{\tau i}\subset B(x,r)\}.\]
Note that $j\geq 2$ and $J\leq m_n-1$. 
We claim that in this case, $x-r\notin I_{\tau(j-1)}$ and $x+r\notin I_{\tau(J+1)}$. 
Suppose first, that $x-r<a_{n_0}:=b(\tau(j-1))$. Then we would have:
\begin{eqnarray*}
\frac{\mu(B(x,r))}{(2r)^t}& 
=& \frac{\mu([x-r,a])+\mu([a(\tau j),x+r])}{(2(a-(x-r))+2(x-a))^t}\\
& \geq& \frac{\mu([x-r,a])}{(2(a-(x-r)))^t}+\frac{\mu([a(\tau j),x+r])}{(2(x+r-a(\tau j)))^t}.
 \end{eqnarray*}
 This would contradict the hypothesis $\liminf_{n\to\infty}\frac{\mu(B(x,r))}{(2r)^t}
<\min\{2^{-t}\alpha^{-1},2^{-t}\beta^{-1}\}$, in view of the previous cases. The proof of $x+r<b_{n_0}$ is similar. 

Moreover, since the center of the ball $(x-r,x+r)$ is in $C$, we can shift $x$ by $d_n$ and still have $x-d_n-r>a_{n_0}$ and $x+d_{n_0}+r<b_{n_0}$. Then, we have:
\[
\frac{\mu(B(x,r))}{(2r)^t} \geq \frac{J-j+1}{\mu(n_0)(b_{n_0}-a_{n_0}-2d_{n_0})^t}\geq\gamma^{-1}-\varepsilon
\]

 \end{proof}
\begin{remark}
 Note that in Lemmas \ref{salto} and \ref{lowbound} the hypothesis $m_n\leq M$ was not used. 
\end{remark}

The proof of the upper bound will be divided into three Lemmas.

\begin{lemma}\label{alpha}
 For almost every $x\in C$, we have $\Theta(\mu,x)\leq 2^{-t}\alpha^{-1}$. 
\end{lemma}

\begin{proof}
Define $k_n$, $\alpha_n$ and $\alpha$ like in the introduction. Pick an increasing sequence $n_j$ such that $\lim_{j\to\infty}\alpha_{n_j}=\alpha$. 
For each $j\geq 1$, define $J:=[\log j/\log M]$. Define the sets:
\[A_j=\{x\in C: \sigma(x)_{n_j-L}=m_{n_j-L}, \sigma(x)_{n_j-L+i}=k_{n_j-L+i} \text{ for }i=1,2,\dots,J\}.\]
Note that $\mu(A_j)=\frac{\mu(n_j-L)}{\mu(n_j-L+J)}\geq M^{-(J+1)}=(jM)^{-1}$ and in consequence, $\sum_j\mu(A_j)=\infty$. 

Now, consider $A=\bigcap_{k\geq j}\bigcup_{j\geq 1}A_j$. If we took $n_j$ sparse enough such that $A_j$ are independent, the Borel Cantelli Lemma implies that $\mu(A)=1$. Therefore it is enough to prove $\Theta(\mu,x)\leq 2^{-t}\alpha^{-1}$ for $x\in A$. 

Given $x\in A$, we have that $x\in A_j$ for infinitely many $j$ and for those $j$ we define:
\[
\rho_j=k_{n_j}s_{n_j}+\sum_{i=1}^{k_{n_j}}g^{n_j}_i-\left( k_{n_j-L+J}s_{n_j-L+J}+\sum_{i=1}^{k_{n_j-L+J}}g_i^{n_j-L+J}\right).
\]
We want to show that $B(x,\rho_j)\cap C=[a(\sigma(x)|n_j),b(\sigma(x)|n_j)]\cap C$ plus -perhaps- some isolated points. 

Put $\sigma:=\sigma(x)$ and $a(\sigma|n_j-L):=a$. We have:
\begin{equation}\label{length}
 x-a\leq k_{n_j-L+J}s_{n_j-L+J}+\sum_{i=1}^{k_{n_j-L+J}}g^{n_j-L+J}_i. 
\end{equation}
So, on one side, we have:
\begin{eqnarray*}
 x+\rho_j&=& x+k_{n_j}s_{n_j}+\sum_{i=1}^{k_{n_j}}g^{n_j}_i-\left( k_{n_j-L+J}s_{n_j-L+J}+\sum_{i=1}^{k_{n_j-L+J}}g_i^{n_j-L+J}\right)\\
 &\leq& a +k_{n_j}s_{n_j}+\sum_{i=1}^{k_{n_j}}g^{n_j}_i=b(\sigma|n_j+g^{n_j}_{k_{n_j}}).
\end{eqnarray*}
On the other side, by the Lemma \ref{salto}:
\begin{eqnarray*}
 x-\rho_j&=& x-k_{n_j}s_{n_j}-\sum_{i=1}^{k_{n_j}}g^{n_j}_i+ k_{n_j-L+J}s_{n_j-L+J}+\sum_{i=1}^{k_{n_j-L+J}}g_i^{n_j-L+J}\\
 &\geq & x-g^{n_j-L}_{n_j-L-1}+ k_{n_j-L+J}s_{n_j-L+J}+\sum_{i=1}^{k_{n_j-L+J}}g_i^{n_j-L+J}\\
 &\geq & a-g^{n_j-L}_{m_{n_{j-L}}-1},
\end{eqnarray*}
Therefore we have one inclusion. To see the other inclusion, note that if $J$ is big enough, the Lemma \ref{salto} also implies:
\begin{equation}\label{star}
 g^{n_j}_{k_{n_j}}\geq k_{n_j-L+J}s_{n_j-L+J}+\sum_{i=1}^{k_{n_j-L+J}}g_i^{n_j-L+J}.
\end{equation}
So, 
\begin{eqnarray*}
x+\rho_j&\geq & x+ k_{n_j}s_{n_j}+\sum_{i=1}^{k_{n_j}-1}g^{n_j}_i\\
&\geq & a+k_{n_j}s_{n_j}+\sum_{i=1}^{k_{n_j}-1}g^{n_j}_i=b(\sigma|n_j).
\end{eqnarray*}
On the other side, using \eqref{length} and \eqref{star} (twice), 
\begin{eqnarray*}
 x-\rho_j&=& x-k_{n_j}s_{n_j}-\sum_{i=1}^{k_{n_j}}g^{n_j}_i+ k_{n_j-L+J}s_{n_j-L+J}+\sum_{i=1}^{k_{n_j-L+J}}g_i^{n_j-L+J}\\
 &\leq & a + 2 k_{n_j-L+J}s_{n_j-L+J}+2 \sum_{i=1}^{k_{n_j-L+J}}g_i^{n_j-L+J}- k_{n_j}s_{n_j}-\sum_{i=1}^{k_{n_j}}g^{n_j}_i \\
 &\leq & a +  k_{n_j-L+J}s_{n_j-L+J}+\sum_{i=1}^{k_{n_j-L+J}}g_i^{n_j-L+J}-\left( k_{n_j}s_{n_j}+\sum_{i=1}^{k_{n_j}-1}g^{n_j}_i \right)\\&\leq& a.
\end{eqnarray*}

Now, \[\frac{\mu(B(x,\rho_j))}{(2\rho_j)^t}= 
\frac{k_{n_j}} {\mu(n_j)2^t\left(k_{n_j}s_{n_j}+\sum_{i=1}^{k_{n_j}}g^{n_j}_i\right)^t\left(1-A\right)^t},
\]
where 
\begin{eqnarray*}
 A & =  &\frac{k_{n_j-L+J}s_{n_j-L+J}+\sum_{i=1}^{k_{n_j-L+J}}g_i^{n_j-L+J}}{k_{n_j}s_{n_j}+\sum_{i=1}^{k_{n_j}}g^{n_j}_i}\\
 &=& \frac{s_{n_j-L+J-1}\left( d_{k_{n_j}-L+J}-d_{k_{n_j}}+r_{n_j-L+J}\right)}
 {s_{n_j}}\\
 &\leq & C 2^{-(J-L-1)}\to_{j\to\infty}0.   
\end{eqnarray*}

\end{proof}

\begin{lemma}
 For almost every $x\in C$, we have $\Theta(\mu,x)\leq 2^{-t}\beta^{-1}$. 
\end{lemma}

\begin{proof}
 The proof is analogous to the proof of Lemma \ref{alpha}
\end{proof}

\begin{lemma}
 If $\gamma^{-1}<2^{-t}\min\{\alpha^{-1},\beta^{-1}\}$ then $\Theta^t(\mu,x)\leq \gamma^{-1}$ for almost every $x\in C$. 
\end{lemma}

\begin{proof}
For those $n$ with $m_n\geq 3$, define $\gamma_n$, $k^1_n,k^2_n$, $a_n,b_n, d_n$ and $\gamma$ like in the introduction.
Put $\varepsilon=\gamma-2^t\max\{\alpha,\beta\}$. If $0<\varepsilon'<\varepsilon$, there are infinitely many $n$ such that $d_n<g^n_i/2-\varepsilon'$ where $i$ is such that $1/2(a_n+b_n)\in g^n_i$. Moreover, we can take $n_j$ such that for all $j$,  $n_j\geq 3$,  $d_{n_j}<g^{n_j}_i/2-\varepsilon'$for all $j$ and $\lim_{j\to\infty} \frac{\mu(n_j)}{k^2_{n_j}-k^1_{n_j}+1}\left(b_{n_j}-a_{n_j}-2 d_{n_j}\right)^t=\gamma$. For each $j\geq 1$ define $J=[\log j/\log M]$ and:
 \[
 A_j=\{x\in C: \sigma(x)|i=\sigma(x_j)|i \text{ for }i=n_j,n_j+1,\dots, n_j+J\}.
 \]
 Like in the proof of the previous Lemma, $A=\bigcap_{k\geq 1}\bigcup_{j\geq k} A_j$ has measure one if $n_j$ are sparse enough by the Borel Cantelli Lemma and therefore is enough to prove the thesis for $x\in A$.
 
 Fix $x\in A$, then $x\in A_j$ for infinitely many $j$ and for those $j$ define
 \(\rho_j=1/2(b_{n_j}-a_{n_j})-d_{n_j}-s_{n_j+J},\). 
 
 We assume $1/2(a_{n_j}+b_{n_j})\geq x_{n_j}$ since the other case is analogous. So, $\rho_{n_j}=x_{n_j}-s_{n_j+J}-a_{n_j}$ and:
 \[x+\rho_{n_j}\leq 2x_{n_j}-s_{n_j+J}-a_{n_j}<b_{n_j}.\]
 On the other side, 
 \[x-r_{n_j}\geq x_{n_j}-s_{n_j+j}-_{n_j}+s_{n_j+J}-a_{n_j}=a_{n_j}.\]
 
 Moreover, $x-\rho_{j}\leq a_{n_j}+g^{n_j}_{i_1}$, where $g^{n_j}_{i_1}$ is the gap next to $a_{n_j}$ and 
 $x+\rho_{j}\geq b_{n_j}-2d_{n_j}-2s_{n_j+J}$ which is bigger than $b_{n_j}-g^{n_j}_{i_2}$ where $g^{n_j}_{i_2}$ is the gap next to $b_{n_j}$ and the estimate is valid for $j$ big enough so that $s_{n_j+J}<\varepsilon'$. 
 
 In conclusion, 
  \[\frac{\mu(B(x,\rho_j)}{(2\rho_j)^t} =\frac{k^2_{n_j}-k^1_{n_j}+1}{b_{n_j}-a_{n_j}-2d_{n_j}(1-A)}, \]
  
  where $A=\frac{s_{n_j+J}}{b_{n_j}-a_{n_j}-2d_{n_j}}\to 0$ when $j\to \infty$. 

\end{proof}

\begin{proof}[Proof of Theorem \ref{packing}]
It is standard to see that $\pac^t(C)<\infty$ (see, for instance \cite{Bae06}, whith the necessary modification). Remember that $\mu=\kappa\pac^t\llcorner_C$ and in particular $\pac^t(C)=\kappa^{-1}$.
Since
$\Theta^t(\pac^t\llcorner_C,x)=1$ for $\pac^t$-almost every $x\in C$ (see Theorem 6.10 in \cite{Mat95}),
we have $\Theta^t(\mu,x)=\kappa \Theta^t(\pac\llcorner_C,x)=\kappa$. So, $\pac^t(C)=\Theta^t(\mu,x)^{-1}$, which implies the thesis in virtue of the previous Lemmas. 
\end{proof}

\begin{remark}
 \begin{enumerate}
  \item If $r_n=1/3$ and $m_n=2$, then $D_n=\{0,2/3\}$. So, $t=\log2/\log3$ and $\alpha_n=\beta_n=2^t=\alpha=\beta.$ Then, $\pac^t(C)=4^t$. This was obtained in \cite{FHW00}.
  \item If $r_n=r$, $m_n=m$ and $D_n=D$ for all $n$, then $C$ is self-similar. This result was previously obtained in \cite{Fen03}, in a more general version (the ratii of contraction need not be all the same).
  \item If $m_n=2$ for all $n$, we obtain the result in \cite{GZ13}.
 \end{enumerate}

\end{remark}


\bibliographystyle{plain}

\end{document}